\documentclass[12pt]{amsart}
\usepackage{amssymb,amsfonts,amsmath,amsopn,amstext,amscd,latexsym, amsthm, enumerate,mathrsfs}

\usepackage{color}
\usepackage{xcolor}
\usepackage[colorlinks=true]{hyperref}
\hypersetup{
     colorlinks   = false,
     citecolor    = blue
}

\usepackage[margin=30mm]{geometry}
\usepackage{bbm}
\usepackage[all,cmtip]{xy}

\usepackage{enumerate}
\usepackage[shortlabels]{enumitem}

\newtheorem{theorem}{Theorem}[section]

\newtheorem{lemma}[theorem]{Lemma}

\newtheorem*{thmA}{Theorem~A}
\newtheorem*{thmB}{Theorem~B}

\newtheorem{lem}[theorem]{Lemma}

\newtheorem{prop}[theorem]{Proposition}

\newtheorem{cor}[theorem]{Corollary}

\theoremstyle{definition}

\theoremstyle{remark}

\DeclareMathOperator{\Aut}{Aut}

\DeclareMathOperator{\Irr}{Irr}

\DeclareMathOperator{\kernel}{Ker}

\newcommand{\bbF}{{\mathbb F}}

\newcommand{\bfZ}{{\mathbf Z}}

\newcommand{\Syl}{{\mathrm {Syl}}}

\DeclareMathOperator{\Sym}{S}

\newcommand{\Fit}{\mathbf{F}}
\newcommand{\OO}{\mathbf{O}}
\newcommand{\Centralizer}{\mathbf{C}}
\newcommand{\Center}{\mathbf{Z}}

\newcommand{\Soc}{\mathrm{Soc}}

\numberwithin{equation}{section}

\newcommand{\Alt}{{\mathrm {A}}}

\newcommand{\TU}{\mathrm{TU}}
\newcommand{\Q}{\mathbb{Q}}
\newcommand{\Z}{\mathbb{Z}}
\begin{document}

\title[Roots of Unity]{Group elements whose character values are roots of unity}

\author[M. Lewis]{Mark L. Lewis}
\address{Department of Mathematical Sciences, Kent State University, Kent, OH 44266, USA}
\email{lewis@math.kent.edu}

\author[L.  Morotti]{Lucia Morotti}
\address{Leibniz Universit\"{a}t Hannover, Institut f\"{u}r Algebra, Zahlentheorie und Diskrete Mathematik, 30167 Hannover, Germany}
\email{morotti@math.uni-hannover.de}

\author[H. P. Tong-Viet]{Hung P. Tong-Viet}

\address{Department of Mathematics and Statistics, Binghamton University, Binghamton, NY 13902-6000, USA}
\email{tongviet@math.binghamton.edu}

\dedicatory{Dedicated to Pham Huu Tiep on his 60th birthday}

\begin{abstract}  We classify all finite groups $G$ which possesses an element $x\in G$ such that every irreducible character of $G$ takes a root of unity value at $x$. 
\end{abstract}

\thanks{}


\date{July 31, 2022}


\maketitle

\section{Introduction}
Let $G$ be a finite group. Following \cite{INW}, an element $x\in G$ is called a nonvanishing element if $\chi(x)\neq 0$ for all irreducible complex characters $\chi$ of $G$. This concept has been widely studied in recent years. In this paper, we consider nonvanishing elements of finite groups which satisfy certain minimal condition as follows. Given a nonvanishing element $x$ of a finite group $G$, it is not hard to show that $|\Centralizer_G(x)|\ge k(G)$ and that the equality holds if and only if $|\chi(x)|=1$ for all irreducible characters $\chi$ of $G$ (see Lemma~\ref{lem:minimal-nonvanishing}), where $\Centralizer_G(x)$ is the centralizer of $x$ in $G$ and $k(G)$ is the number of conjugacy classes of $G$. Note that if $|\chi(x)|=1$ for some character $\chi$ of $G$, then $x$ is a root of unity (see, for example, Problem 3.2 in \cite{Isaacs}). We will call an element $x\in G$ a \emph{root of unity element} if $|\chi(x)|=1$ for all irreducible characters of $G$. The condition $|\Centralizer_G(x)|=k(G)$ alone does not characterizes root of unity elements. For example, if $G=\Alt_5$, the alternating group of degree $5$, and $x\in G$ is an element of order $5$, then $|\Centralizer_G(x)|=5=k(G)$ but $x$ is not a root of unity element.
We note that root of unity elements are called totally unitary or $\TU$-elements by  S. Ostrovskaya and E. M. Zhmud' and they classify all finite metabelian groups with trivial center that contain a root of unity element in \cite[Chapter XXII]{BKZ}.

Write $\Irr(G)$ for the set of irreducible complex characters of $G$ and  $\Fit(G)$ for the Fitting subgroup of $G$, that is, the largest normal nilpotent subgroup of $G$. In our first result, we prove the following.
\begin{thmA}
Let $G$ be a finite group and let $x\in G$. If $|\chi(x)|=1$ for all irreducible characters $\chi$ of $G$, then  $x\in \Fit(G)$ and both $\Fit(G)$ and $G/\Fit(G)$ are abelian. In particular, $G$ is abelian or metabelian,
\end{thmA}

Thus if a finite group $G$ has a root of unity element, then it is abelian or metabelian. In particular, such a group is solvable.
Theorem A confirms a conjecture  proposed in  \cite{INW} for root  of unity elements. This conjecture states that every nonvanishing element of a finite solvable groups $G$ must lie in the Fitting subgroup $\Fit(G)$. 

Write $\Irr(G)$ for the set of all complex irreducible characters of $G$. Our interest in root of unity elements stems from an observation that if $\chi\in\Irr(G)$ and $g\in G$ such that $|\chi(g)|=1$, then the size of the conjugacy class $g^G$ containing $g$ is always divisible by $\chi(1)$ (see Lemma~\ref{lem:divisibility}). Consequently, if $x$ is a root of unity element of $G$, then  $|x^G|$ is divisible by $\chi(1)$ for all $\chi\in\Irr(G)$. This is  related to Conjecture C in \cite{IKMM} asserting that if $\chi\in\Irr(G)$ is a primitive character of a finite group $G$, then $\chi(1)$ divides $|g^G| $ for some $g\in G.$ Thus the observation above gives us a way to locate the required element $g\in G$. However, not every primitive irreducible character admits a root of unity value. For example,  if $G$ is the sporadic simple group $\textrm{O'N}$, then $G$ has a primitive irreducible  character of degree $64790$ which does not admit any root of unity value.

In the next result, we classify all finite groups with a root of unity element. Clearly, if $G$ is abelian, then every element of $G$ is a root of unity element. Let $q>2$ be a prime power. We denote by $\Gamma_q$ the unique doubly transitive Frobenius group with a cyclic complement of order $q-1$ and degree $q.$ So $\Gamma_q\cong \textrm{AGL}_1(\bbF_q)=\bbF_q\rtimes \bbF_q^*,$ where $\bbF_q$ is a finite field with $q$ elements.

\begin{thmB} Let $G$ be a finite group. Then $G$ has a root of unity element $x\in G$ if and only if one of the following holds:
\begin{itemize}
\item $G$ is abelian;
\item $\Fit(G)=G'\Center(G)$ is abelian, $G'\cap \Center(G)=1;$ and $G/\Center (G)\cong \Gamma_{q_1}\times \Gamma_{q_2}\times\cdots\times \Gamma_{q_m}$, where each $q_i>2$ is a prime power and $m\ge 1$ is an integer.
\end{itemize}
\end{thmB}

We also show that if a finite group $G$ has a root of unity element, then it is an $A$-group, that is, $G$ is solvable and all its Sylow subgroups are abelian.
For each $i$ with $1\leq i\leq m,$ write $\Gamma_{q_i}=V_i\ltimes A_i$, where $V_i$ is the Frobenius kernel and $A_i$ is the Frobenius complement. Let  $U:=\prod_{i=1}^m (V_i-\{1\})$ and let $\mathcal{U}=\pi^{-1}(U)$ where $\pi:G\rightarrow G/\Center(G)$ is the natural homomorphism. Then every element of $U$ is a root of unity element of $G/\Center(G)$ by \cite[Lemma 3.17]{BKZ} and from the proof of Theorem B, every element of $\mathcal{U}$ is a root of unity element of $G$.

Our notation is standard and we follow \cite{Isaacs} for the character theory of finite groups.

\section{Preliminaries}

 We collect some properties of root of unity elements in the next lemmas.

\begin{lemma}\label{lem:divisibility}
Let $G$ be a finite group and let $g\in G$. If $\chi\in\Irr(G)$ and $|\chi(g)|=1$, then $\chi(1)$ divides $|g^G|.$ In particular, if $x\in G$ is a root of unity element, then $\chi(1)$ divides $|x^G|$ for all $\chi\in\Irr(G)$.
\end{lemma}

\begin{proof}  Assume that $\chi\in\Irr(G)$ and $g\in G$ such that $|\chi(g)|=1.$ Let $K$ be the class sum of the conjugacy class $g^G,$ that is, $K=\sum_{y\in g^G}y.$ Then $$\omega_\chi(K)=\frac{|g^G|\chi(g)}{\chi(1)}$$ is an algebraic integer by \cite[Theorem 3.7]{Isaacs}. Similarly, \[\omega_{\overline{\chi}}(K)=\frac{|g^G|\overline{\chi}(g)}{\overline{\chi}(1)}=\frac{|g^G|\overline{\chi(g)}}{\chi(1)}\] is also an algebraic integer, where $\overline{\chi}$ is the complex conjugate of $\chi$. Since the products of algebraic integers are algebraic integers, we  see that \[\omega_\chi(K)\omega_{\overline{\chi}}(K)=\frac{|g^G|^2|\chi(g)|^2}{\chi(1)^2}=\frac{|g^G|^2}{\chi(1)^2}\] is an algebraic integer.  Clearly, ${|g^G|^2}/{\chi(1)^2}$ is a rational number,  so $|g^G|^2/\chi(1)^2$ is an integer which implies that $\chi(1)$ divides $|g^G|$ as wanted.

If  $x\in G$ is a root of unity element, then for any $\chi\in\Irr(G)$, we have $|\chi(x)|=1$ and hence $\chi(1)$ divides $|x^G|$ as wanted.
\end{proof}

\begin{lem}\label{lem:properties} Let $G$ be a finite group and let $x\in G$ be a root of unity element. Then
\begin{enumerate}[$(a)$]

\item $k(G)=|\Centralizer_G(x)|.$
\item  $G'\leq \langle x^G\rangle.$
\item If $N\unlhd G$, then $xN$ is a root of unity in $G/N.$
\item If $z\in \Center(G)$, then $xz$ is also a root of unity element.
\end{enumerate}

\end{lem}

\begin{proof}
 From the Second Orthogonal relation, we have \[|\Centralizer_G(x)|=\sum_{\chi\in\Irr(G)}|\chi(x)|^2=\sum_{\chi\in\Irr(G)}1=k(G).\]

Let $L=\langle x^G\rangle.$ For any $\chi\in\Irr(G/L)$, we see that  $x\in L\subseteq \kernel(\chi).$ Hence $1=|\chi(x)|=\chi(1)$ and thus all characters $\chi\in\Irr(G/L)$ are linear which implies that $G/L$ is abelian and so $G'\leq L.$  Since $\Irr(G/N)\subseteq\Irr(G)$ whenever $N\unlhd G$, if $x$ is a root of unity of $G$ then $xN$ is a root of unity of $G/N.$ 

Finally, let $z\in \Center(G)$ and $\chi\in\Irr(G)$. Then $\chi_{\Center(G)}=\chi(1)\lambda$ for some $\lambda\in\Irr(\Center(G)).$ We have $\chi(xz)=\lambda(z)\chi(x)$ and thus if $x$ is a root of unity element, then so is $xz$ as $|\lambda(z)|=1.$
\end{proof}

The next lemma follows from the proof of Lemma 3.17 in \cite[Chapter XXII]{BKZ} and the previous lemma. For completeness, we include the proof here.
\begin{lem}\label{lem:minimal-nonvanishing} Let $G$ be a finite group and let $x\in G$ be a nonvanishing element of $G$. Then $|\Centralizer_G(x)|\ge k(G)$; and the equality holds if and only if $x$ is a root of unity element.
\end{lem}

\begin{proof}
We first claim that if $x\in G$ is a nonvanishing element, then $|\Centralizer_G(x)|\ge k(G)$. Let $\alpha=\prod_{\chi\in\Irr(G)}\chi(x)$. Let $n$ be the exponent of $G$ and let $\Q_n=\Q(\xi)$, where $\xi$ is a primitive $nth$-root of unity. Let $\mathcal{G}$ be the Galois group of $\Q_n$ over $\Q.$ Then $\mathcal{G}$ acts on $\Irr(G)$ and we see that $\chi^\sigma\in\Irr(G)$ if and only if $\chi\in\Irr(G)$ for all $\sigma\in\mathcal{G}$. Hence $\alpha^\sigma=\alpha$ for all $\alpha\in\mathcal{G}$. It follows that $\alpha\in\Q$. Since $\alpha$ is an algebraic integer, we must have that $\alpha\in\Z.$ As $x$ is nonvanishing, $\alpha\neq 0$ and so $|\alpha|\ge 1.$

By the inequality between arithmetic and geometric means, we have that \[1\leq |\alpha|^2=\prod_{\chi\in\Irr(G)}|\chi(x)|^2\leq \left(\dfrac{\sum_{\chi\in\Irr(G)}|\chi(x)|^2}{k(G)}\right)^{k(G)}=\left(\dfrac{|\Centralizer_G(x)|}{k(G)}\right)^{k(G)}.\]
 It follows that $|\Centralizer_G(x)|\ge k(G)$ as wanted.

Next, assume that $x$ is a nonvanishing element of $G$ and $|\Centralizer_G(x)|=k(G)$.  Then $|\alpha|=1$ from the inequality above. Hence \[\prod_{\chi\in\Irr(G)}|\chi(x)|^2= \left(\dfrac{\sum_{\chi\in\Irr(G)}|\chi(x)|^2}{k(G)}\right)^{k(G)}=1.\] Therefore, $|\chi(x)|=1$ for all $\chi\in\Irr(G).$ So $x\in G$ is a root of unity element. 

Conversely, if $x$ is a root of unity, then clearly $x$ is nonvanishing and $|\Centralizer_G(x)|=k(G)$ by Lemma \ref{lem:properties}(a).  
\end{proof}

A consequence of the previous lemma is that if $x\in G$ and  $k(G)>|\Centralizer_G(x)|$ or equivalently $|x^G|>|G|/k(G)$,  the average of the conjugacy class size of $G$, then $x$ is a vanishing element of $G$, that is, $\chi(x)=0$ for some $\chi\in\Irr(G)$. Also, if $G$ has a root of unity element $x$, then the commuting probability $$\textrm{cp}(G)=\dfrac{|\{(a,b)\in G\times G: ab=ba\}|}{|G|^2}=\frac{k(G)}{|G|}$$ is equal to $1/|x^G|.$

In the next two lemmas, we quote some results in Chapter XXII of \cite{BKZ}.
\begin{lem}\label{lem:metabelian}
Let $G$ be a finite group and suppose that $x\in G$ is a root of unity element. 
\begin{enumerate}[$(a)$]
\item If $x\in \Fit(G)$, then $\Fit(G)$ is abelian and $G'\leq \Fit(G)$. In particular, $G$ is abelian or metabelian.
\item Conversely, if $G$ is metabelian, then $x\in \Fit(G)$.
\end{enumerate}
\end{lem}

\begin{proof} 
This is Lemma 1.5 in \cite[Chapter XXII]{BKZ}.
\end{proof}

The following  is the main result of Chapter XXII in \cite{BKZ}. Recall that the socle of a finite group $G$, denoted by $\textrm{Soc}(G)$, is a product of all minimal normal subgroups of $G$.

\begin{lem}\label{lem:metabelian with trivial center}
Let $G$ be a finite metabelian group with trivial center. Then $G$ has a root of unity element if and only if $G\cong \Gamma_{q_1}\times \Gamma_{q_2}\times \cdots\times \Gamma_{q_m}$, where $q_1,q_2,\cdots, q_m$ are prime power $>2.$ Moreover, if $x\in G$ is a root of unity, then $\Fit(G)={\Soc}(G)=G'=\langle x^G\rangle$ and $\Centralizer_G(x)=\Fit(G)$.
\end{lem}

\begin{proof}
The equivalent statements follow from Theorem 1.12 and Corollary 1.11 and the last claim follows from Lemmas 3.8 and 3.14 in \cite[Chapter XXII]{BKZ}.
\end{proof}

For a finite group $G$, recall that 
$\Fit(G)$, the Fitting subgroup of $G$, is the largest nilpotent normal subgroup of $G$. The Fitting series of a finite group $G$ is defined by $\Fit_1(G):=\Fit(G)$ and for any integer $i\ge 1,$  $\Fit_{i+1}(G)/\Fit_i(G)=\Fit(G/\Fit_i(G))$. Similarly, the upper central series of $G$ is defined by $\bfZ_1(G):=\Center(G)$ and for $i\ge 1,$ we have $\bfZ_{i+1}(G)/\bfZ_i(G)=\Center(G/\bfZ_i(G)).$  The last term of the upper central series of $G$ is called the hypercenter (or hypercentral) of $G$ and is denoted by $\Center_\infty(G)$. 

The following results are well-known.
\begin{lem}\label{lem:Fitting} Let $G$ be a finite group and let $N$ be a normal subgroup of $G$ such that $N\leq \Center(G)$.

\begin{enumerate}[$(1)$]
\item If $G/N$ is nilpotent, then $G$ is nilpotent.

\item $\Fit(G/N)=\Fit(G)/N.$

\item $\Fit(G/\bfZ_i(G))=\Fit(G)/\bfZ_i(G)$ for all $i\ge 1$.
\end{enumerate}

\end{lem}

\begin{proof}
The first two claims are well-known. The last claim follows from the second claim and induction.
\end{proof}

The next result is Corollary 2.3 in \cite{MW}.
\begin{lem}\label{lem:abelian Fitting}
Let $G$ be a finite solvable group and assume that the Sylow $2$-subgroups of $\Fit_{i+1}(G)/\Fit_i(G)$ are abelian for $1\leq i\leq 9$. Then every nonvanishing element of $G$ lies in $\Fit(G).$
\end{lem}

\section{Solvability of finite groups with a root of unity element}

We first prove Theorem A for finite solvable groups.

\begin{prop}\label{prop:metabelian}
Let $G$ be a finite solvable group and suppose that $x\in G$ is a root of unity element. Then $G$ is  abelian or $G$ is metabelian, $x\in\Fit(G)$ and both $\Fit(G)$ and $G/\Fit(G)$ are abelian. 
\end{prop}

\begin{proof} 
If $G$ is abelian, then we are done. So assume that $G$ is nonabelian. If $G$ is metabelian, then the conclusion follows from Lemma \ref{lem:metabelian}. Thus we only need to show that $G$ is metabelian. We will prove this by induction on $|G|$.

Let $N$ be a minimal normal subgroup of $G$. Let  $\overline{G}=G/N$ and use the `bar' notation. By Lemma \ref{lem:properties}\:(c), $\bar{x}$ is a root of unity element in $\overline{G}$ and thus by induction, $\overline{G}$ is  abelian or metabelian. If $\overline{G}$ is abelian, then $G$ is metabelian. So assume that  $\overline{G}$ is metabelian (but $G$ is neither metabelian nor abelian). Then  $\bar{x}\in\Fit(\overline{G})$ and both $\Fit(\overline{G})$ and $\overline{G}/\Fit(\overline{G})$ are abelian by Lemma \ref{lem:metabelian}. Further $N$ is abelian since it is a minimal normal subgroup and $G$ (and so also $N$) is solvable. Thus $N\leq\Fit(G)$.

Therefore, $G/\Fit({G})$ is also metabelian. Again by Lemma \ref{lem:metabelian} we have that $\Fit(G/\Fit(G))=\Fit_2(G)/\Fit_1(G)$ and $G/\Fit_2(G)$ are abelian. 
It follows that  $\Fit_3(G)=G$ and $\Fit_{i+1}(G)/\Fit_i(G)$ is abelian for all $i\ge 1$. Now by Lemma \ref{lem:abelian Fitting}, we have $x\in\Fit(G)$ as $x$ is nonvanishing and so $G$ is metabelian. This contradiction completes the proof.
\end{proof}

The following result follows from  the proof of Theorem A in \cite{MT1}.  Recall that an element $g\in G$ is called a vanishing element if $\chi(g)=0$ for some $\chi\in\Irr(G).$

\begin{lem}\label{lem:vanishing} Let $G$ be a finite group. Assume that $G$ has a unique minimal normal subgroup $N$. If $N$ is non-abelian and $G/N$ is solvable, then every element in $G-N$ is a vanishing element.
\end{lem}

Let $n\ge 2$ be an integer and let $\lambda=(\lambda_1,\lambda_2,\dots,\lambda_r)$ be a partition of $n.$ For $1\leq i\leq k$ and $1\leq j\leq\lambda_i$, we denote by $h_{i,j}^\lambda$ the hook length of the node $(i,j)$ of the Young diagram of $\lambda$. Let $\lambda$ and $\mu$ be partitions of $n.$ We use the notation $\chi_\mu^\lambda$ to denote the value of the irreducible character of $\Sym_n$ labeled by $\lambda$ evaluated at the conjugacy class with cycle type $\mu$. 

\begin{lem}\label{lem:alternating}
Let $n\ge 6$ be an integer and let $x\in\Alt_n$. Then there exists a partition $\lambda$ of $n$ which is not self-conjugate such that $\chi^\lambda(x)\neq \pm 1.$
\end{lem}

\begin{proof}
We will use the following fact  which can follow easily from Murnaghan-Nakayama formula. If $m\ge 1$ is an integer and $\gamma=(\gamma_1,\gamma_2,\dots,\gamma_r),$ $\beta=(\beta_1,\beta_2,\dots,\beta_s)$ are partition of $m$ with $h_{2,1}^{\gamma}<\beta_1$ and $\gamma_1-\gamma_2\ge \beta_1$, then $$\chi^\gamma_\beta=\chi^{(\gamma_1-\beta_1,\gamma_2,\dots,\gamma_r)}_{(\beta_2,\beta_3,\dots,\beta_2)}.$$

Let $\alpha\vdash n$ be the cycle partition of $x\in\Alt_n.$ Since $n\ge 6,$  from the proof of Lemma 1.6 in \cite{Mo} we may assume that all parts of $\alpha$ are distinct, except possibly for the part $1$, which may have multiplicity $2$.

Write $\alpha=(\alpha_1,\alpha_2,\dots,\alpha_l)\vdash n.$ We consider the following cases.

\textbf{Case $1:$} $l\ge 2$ and $\alpha_{l-1}=\alpha_l.$ In this case, we have $\alpha_{l-1}=\alpha_l=1$, $l\ge 3 $ (as $n\ge 6$) and $\alpha_{l-2}>1$.

Assume first that $\alpha_{l-2}>2$. Since $n\ge 6,$ the partition $(n-2,1,1)$ of $n$ is not self-conjugate and we have that \[\chi_{\alpha}^{(n-2,1,1)}=\chi^{(\alpha_{l-2},1,1)}_{(\alpha_{l-2},1,1)}=0.\] The  first equality holds by the observation above and the latter equality holds since $\alpha_{l-2}>2$ so 
$$\begin{array}{ccl} h_{1,1}^{(\alpha_{l-2},1,1)}&=&\alpha_{l-2}+2>\alpha_{l-2}\\
h_{1,2}^{(\alpha_{l-2},1,1)}&=&\alpha_{l-2}-1<\alpha_{l-2}\\
h_{2,1}^{(\alpha_{l-2},1,1)}&=&2<\alpha_{l-2}.
\end{array}$$

 Assume next that $\alpha_{l-2}=2.$ Then $l\ge 4$ and $\alpha_{l-3}\ge 3$. Then the partition $(n-2,2)$ of $n$ is not self-conjugate and by the observation above, we have \[\chi_{\alpha}^{(n-2,2)}=\chi^{(2,2)}_{(2,1,1)}=0.\]

 \textbf{Case $2:$} $l\ge 2$ and $\alpha_{l-1}\neq \alpha_l.$
 
 Assume first that $l\ge 3$ or $\alpha_{l-1}\neq \alpha_l+1.$ Then
 
  \[\chi^{(n-\alpha_l,1^{\alpha_l})}_{\alpha}=\chi^{(\alpha_{l-1},1^{\alpha_l})}_{(\alpha_{l-1},\alpha_l)}=0\] 
  as 
 $$\begin{array}{ccl} h_{1,1}^{(\alpha_{l-1},1^{\alpha_l})}&=&\alpha_{l-1}+\alpha_l>\alpha_{l-1}\\
h_{1,2}^{(\alpha_{l-1},1^{\alpha_l})}&=&\alpha_{l-1}-1<\alpha_{l-1}\\
h_{2,1}^{(\alpha_{l-1},1^{\alpha_l})}&=&\alpha_l<\alpha_{l-1}
\end{array}$$
 
 and  $(n-\alpha_l,1^{\alpha_l})'=(\alpha_l+1,1^{n-\alpha_l-1})\neq (n-\alpha_l,1^{\alpha_l})$ as
 $$n-\alpha_l\ge \alpha_{l-2}+\alpha_{l-1}\ge 2\alpha_l+3>\alpha_l+1$$ if $l\ge 3$; and $n-\alpha_l\ge \alpha_{l-1}>\alpha_l+1$ if $\alpha_{l-1}\neq\alpha_l.$
 
 Assume next that $\ell=2$ and $\alpha_1=\alpha_2+1.$ Then $n=2\alpha_2+1.$ As $n\ge 6,$ we have $\alpha_2\ge 3.$ Again $(n-2,2)'\neq (n-2,2)$.
 
If $\alpha_2=3,$ then $n=7$ and $\chi^{(5,2)}_{(4,3)}=0.$ Assume that $\alpha_2\ge 4.$ Then  \[\chi_{\alpha}^{(n-2,2)}=\chi^{(n-2,2)}_{(\alpha_{2}+1,\alpha_2)}=\chi^{(\alpha_2-2,2)}_{(\alpha_2)}=0\] as $h_{1,1}^{(\alpha_2-2,2)}=\alpha_2-1<\alpha_2.$

 \textbf{Case $3:$} $l=1$. Then $\alpha=(n).$ Since $n\ge 6$, $(n-2,2)$ is not self-conjugate and $\chi^{(n-2,2)}_{(n)}=0$ as $h_{1,1}^{(n-2,2)}=n-1<n.$
\end{proof}

We are now ready to prove Theorem A.
\begin{proof}[\textbf{Proof of Theorem A}] Let $x\in G$ be a root of unity element. If $G$ is solvable, then the theorem follows from Proposition \ref{prop:metabelian}. Thus it suffices to show that $G$ is solvable. Suppose not and let $G$ be a counterexample to the theorem with $|G|$ minimal. Then  $x\in G$ is a root of unity but  $G$ is non-solvable. Let  $L=\langle x^G\rangle.$ Then $G'\leq L\unlhd G$ by Lemma \ref{lem:properties}\:(b).

Let $N$ be a minimal normal subgroup of $G$. By Lemma \ref{lem:properties}\:(c), $G/N$ has a root of unity $xN$. Since $|G/N|<|G|$, $G/N$ is solvable.  As $G$ is non-solvable,  $N$ is non-solvable.  If $G$ has two distinct minimal normal subgroups, say $N_1\neq N_2$, then $N_1\cap N_2=1$ and thus $G$ embeds into $G/N_1\times G/N_2$, where the latter group is solvable by the argument above. Therefore, $G$ is solvable, which is a contradiction. It follows that $G$ has a unique minimal normal subgroup $N$, which is non-solvable and $G/N$ is solvable. Hence $N=S_1\times S_2\cdots\times S_k$, where each $S_i\cong S$ for some non-abelian simple group $S$.  By Lemma \ref{lem:vanishing}, $x\in N$ as every element in $G-N$ is a vanishing element. It follows from Lemma \ref{lem:properties}\:(b) that $G'=\langle x^G\rangle=N,$ and so $G/N$ is abelian.

Write $x=(x_1,x_2,\dots,x_k)\in N$, where $x_i\in S_i$ for $1\leq i\leq k.$ As $G$ is nonabelian, $x$ is nontrivial and so $o(x)$, the order of $x$, is divisible by some prime $p\ge 2.$ Clearly, $o(x_i)$ is divisible by $p$ for some $i\ge 1.$ Assume that $S$ has an irreducible character $\theta$ of $p$-defect zero. Then $\lambda=\theta\times\theta\times\cdots\times \theta\in\Irr(N)$ has $p$-defect zero. Clearly, every $G$-conjugate of $\lambda$ also has $p$-defect zero and hence if $\chi\in\Irr(G)$ lying over $\lambda$, then $\chi_N$ is a sum of $G$-conjugates of $\lambda$ so that $\chi(x)=0$ since every conjugate of $\lambda$ vanishes at $x$ as $o(x)$ is divisible by $p$. Therefore, we can assume that $S$ has no $p$-defect zero character. By \cite[Corollary~2]{GO}, the following cases hold.

\begin{enumerate}
\item $p=2$ and $S$ is isomorphic to $\textrm{M}_{12},\textrm{M}_{22},\textrm{M}_{24},$ $\textrm{J}_{2}, \textrm{HS},\textrm{Suz},\textrm{Ru}$, $\textrm{Co}_1,\textrm{Co}_{3}$ or $\Alt_n$ for some integer $n\ge 7$; or

\item $p=3$ and $S$ is isomorphic to $\textrm{Suz},\textrm{Co}_{3}$ or $\Alt_n$ for some integer $n\ge 7$, 
\end{enumerate}

We make the following observation. Assume that $S$ has a rational-valued irreducible character $\theta\in\Irr(S)$ which is extendible to $\Aut(S)$. Then $\varphi=\theta\times\theta\times\cdots\times \theta\in\Irr(N)$ extends to $\chi\in\Irr(G)$ and $$1=|\chi(x)|=|\varphi(x)|=\prod_{i=1}^k|\theta(x_i)|.$$ Since $\theta$ is rational, $\theta(x_i)$ is a non-zero integer and thus $|\theta(x_i)|\ge 1$ for all $i$. The previous equation now  implies that $|\theta(x_i)|=1$ for all $i.$

(a) Assume first that $S$ is one of the sporadic simple groups in (i) but not in (ii). Then $x$ and hence $x_i's$ must be a $2$-element. Using \cite{GAP}, we can find an irreducible rational-valued character $\theta$ which is extendible to $\Aut(S)$ and does not take root of unity on any $2$-elements. So this case cannot occur.

Similarly, if $S$ appears in Case (ii), then $x$ and hence $x_i's$ are  $\{2,3\}$-elements. Again, by using \cite{GAP}, we we can find an irreducible rational-valued character $\theta$ which is extendible to $\Aut(S)$ and does not take root of unity on any $\{2,3\}$-elements. 

(b) Assume that $S\cong \Alt_n$, where $n\ge 7$ is an integer and that $\Alt_n$ has no block of $p$-defect zero for $p=2,3$ or both.

By the observation above, if $\lambda$ is a partition of $n$ which is not self-conjugate, then $\chi^{\lambda}$, the irreducible character of $\Sym_n$ labeled by $\lambda$, remains irreducible upon reduction to $\Alt_n$ and thus $|\chi^{\lambda}(x_i)|=1.$ Note that $\chi^{\lambda}$ is rational-valued. Now Lemma \ref{lem:alternating} provides a contradiction.

Therefore, $G$ must be solvable as wanted. The proof is now complete.
\end{proof}

\section{Finite metabelian groups with a root of unity element}
In this section, we will characterizing finite metabelian groups with  a root of unity element. Such a group with trivial center was classified by S. Ostrovskaya and E. M. Zhmud'. Recall that if $q>2$ is a prime power, then $\Gamma_q$ is a doubly transitive Frobenius group with a cyclic complement of order $q-1$ and degree $q$. Note that $\Gamma_q$ has a root of unity element and every root of unity element of $\Gamma_q$ lies in $\Fit(\Gamma_q)$, which is an elementary abelian $p$-group, where $q$ is a power of a prime $p$. Moreover, if $\Gamma=\Gamma_{q_1}\times \Gamma_{q_2}\times\cdots\times\Gamma_{q_m}$, where each $q_i>2$ are prime powers, then $\Gamma$ has a root of unity element and furthermore, all Sylow subgroups of $\Gamma$ are abelian.

\begin{lem}\label{lem4.1}
Let $G$ be a finite group and let $x\in G$ be a root of unity element. Let $K=\Center_\infty(G)$ be the hypercenter of $G$. Assume that $G$ is nonabelian. Then

\begin{enumerate}[$(1)$]
\item $G/K\cong \Gamma_{q_1}\times \Gamma_{q_2}\times \cdots\times \Gamma_{q_m}$, where each ${q_i}>2$ is a prime power and $m\ge 1$ is an integer. Moreover, $\Fit(G)=G'K$ is abelian, and $\Centralizer_{G/K}(xK)=\Fit(G/K)=\Fit(G)/K.$
\item $\Fit(G)=\Centralizer_G(x)$ and $k(G)=|\Fit(G)|=|\Centralizer_G(x)|$.

\item If $N=\bfZ_i(G)$ for some $i\ge 1$ or $N\leq \Center(G)$, then $\Centralizer_{G/N}(xN)=\Fit(G/N)=\Fit(G)/N$ and $k(G/N)=|\Fit(G):N|.$

\end{enumerate}
\end{lem}

\begin{proof}  By Theorem A, $x\in F:=\Fit(G)$, $F$ is abelian and $G'\leq F.$ 
 Since $K\unlhd G$ is nilpotent, we have $\Center(G)\leq K\leq F.$ Now the center of $G/K$ is trivial by the definition of $K$. Moreover $\Fit(G/K)=F/K$ by Lemma \ref{lem:Fitting}\:(3). Since $G/K$ has a root of unity element $xK$, part (1) follows from Lemma  \ref{lem:metabelian with trivial center}. 

Since $x\in F$ and $F$ is abelian, we have $K\leq F\leq \Centralizer_G(x)$. Let $\overline{G}=G/K$.  From part (1), we have $\Centralizer_{\overline{G}}(\overline{x})=\Fit(\overline{G})=\overline{F}$. Hence \[\overline{F}\leq \overline{\Centralizer_G(x)}\leq \Centralizer_{\overline{G}}(\overline{x})=\overline{F}.\] Thus $\overline{F}=\overline{\Centralizer_G(x)}$ and hence $F=\Centralizer_G(x)$. As $k(G)=|\Centralizer_G(x)|$ by Lemma \ref{lem:properties}\:(a), part (2) follows.

Finally, let $N=\Center_i(G)$ for some $i\ge 1$ or $N\leq \Center(G).$ Then $G/N$ has a root of unity and it is not nilpotent. By Lemma \ref{lem:Fitting}, $\Fit(G/N)=F/N$. Now part (3) follows by applying part (2) to $G/N$.
  \end{proof}

Following P. Hall, a finite solvable group $G$ is called an $A$-group if every Sylow subgroup of $G$ is abelian. The next lemma shows that any finite group with a root of unity element is an $A$-group. 

\begin{prop}\label{prop:A-groups}
If a finite group $G$ has a root of unity element, then $G$ is an $A$-group.
\end{prop}

\begin{proof}
Let $G$ be a finite group with a root of unity element $x\in G$. Clearly, if $G$ is abelian, then $G$ is an $A$-group. So, we can assume that $G$ is nonabelian and hence by Theorem A,  $x\in F:=\Fit(G)$ and both $F$ and $G/F$ are abelian so $G$ is solvable. We proceed by induction on $|G|$ that all Sylow subgroups of $G$ are abelian. 

Notice first that if $1<N\unlhd G$, then $G/N$ has a root of unity element $xN$ and so by induction, every Sylow subgroup of $G/N$ is abelian. Now let $N$ be a minimal normal subgroup of $G$. Since $G$ is solvable, $N$ is an elementary abelian $p$-group for some prime $p.$ If $Q$ is a Sylow $r$-subgroup of $G$ for some prime $r\neq p$, then $QN/N\cong Q/Q\cap N\cong Q$ is abelian. Thus it remains to show that every Sylow $p$-subgroup of $G$ is abelian. Let $P\in\Syl_p(G)$. Then $N\unlhd P$  and $P/N$ is abelian as $P/N\in\Syl_p(G/N)$. Hence $P'\leq N.$ Now if $G$ has another minimal normal subgroup, say $M\neq N$, then $P'\leq N\cap M=1$ and hence $P$ is abelian as wanted if $M$ is also a $p$-group. If instead $M$ is not a $p$-group then we can conclude as in the $r\not=p$ case that $P$ is abelian. Therefore, we may assume that $N$ is the unique minimal normal subgroup of $G$.

Since $F$ is abelian and $N\leq F$ is the unique minimal normal subgroup of $G$, $F$ must be a $p$-group and so $F=\OO_p(G)$. Let $K=\Center_\infty(G)$. By Lemma \ref{lem4.1}, $K\leq F\unlhd G$ and $G/K\cong \Gamma_{q_1}\times \Gamma_{q_2}\times\cdots\times\Gamma_{q_m}$ for some integer $m\ge 1$ and each $q_i>2$ is a prime power. Write $\Gamma_{q_i}=V_i\rtimes A_i$, where $V_i$ is the Frobenius kernel which is an elementary abelian group of order $q_i$ and $A_i$ is cyclic of order $q_i-1$. Now $F/K=\Fit(G/K)=V_1\times V_2\times\cdots\times V_m$ is a direct product of the Frobenius kernels of  the groups $\Gamma_{q_i}$. It follows that each $q_i$ is a power of $p$. Hence $p\nmid |A_i|$ for all $i$ and so $|G:F|=\prod_{i=1}^m|A_i|=\prod_{i=1}^m(q_i-1)$ is coprime to $p$. Therefore, $F$ is a Sylow $p$-subgroup of $G$ which is abelian and we are done.
\end{proof}

\begin{cor}\label{cor1} Let $G$ be a finite group and let $x\in G$ be a root of unity. Then $\Center(G)=\Center_\infty(G)$, that is, $G/\Center(G)$ has a trivial center, and $G'\cap \Center(G)=1$.
\end{cor}

\begin{proof}
By Proposition \ref{prop:A-groups}, $G$ is a finite solvable $A$-group. The result now follows from $(3.8)$ and Theorem 4.1 in \cite{Taunt}.
\end{proof}


\begin{proof}[\textbf{Proof of Theorem B}]
Let $G$ be a finite group, let $Z:=\Center(G)$ and $F:=\Fit(G)$. Suppose first that $x\in G$ is a root of unity element. If $G$ is abelian, then we are done. Assume that $G$ is non-abelian. By Theorem A, $x\in F$, $F$ and $G/F$ are abelian and $G$ is metabelian. By Corollary \ref{cor1}, $\Center_\infty(G)=\Center(G)$, $G/Z$ has trivial center and $G'\cap \Center(G)=1$. Now,  the conclusion follows from Lemma \ref{lem4.1}\:(1). 

For the converse, assume that $G$ is nonabelian. So  $F=G'Z$ is abelian, $G'\cap Z=1$ and $G/Z\cong \prod_{i=1}^m\Gamma_{q_i}$ for some integer $m\ge 1$ and prime powers $q_i>2$. In particular, $G$ is a metabelian group. Write $\overline{G}=G/Z$ and use the `bar' notation. By Lemma \ref{lem:metabelian with trivial center}, $\overline{G}$ has a root of unity element $\overline{x}$ for some $x\in G$. Note that the hypothesis above implies that $F=G'\times Z.$


We claim that $x$ is also a root of unity element of $G$, that is, $|\chi(x)|=1$ for all $\chi\in\Irr(G)$. As $|\chi(x)|=|\chi(\overline{x})|=1$ for every $\chi\in\Irr(G/Z)$, it suffices to show that if $1\neq \lambda\in\Irr(Z)$ and $\chi\in\Irr(G)$ lying over $\lambda$, then $|\chi(x)|=1.$

Let $1\neq \lambda\in\Irr(Z)$. Since $F=Z\times G'$, $\theta=\lambda\times 1_{G'}\in\Irr(F)$ is an extension of $\lambda$ and $G'\leq \kernel(\theta)$. So $\theta$  can be considered an irreducible character of $F/G'$ and thus $\theta$ extends to $\phi\in\Irr(G/G')$. Thus $\lambda$ extends to $\phi\in\Irr(G)$. By Gallagher's theorem, every $\chi\in\Irr(G)$ lying above $\lambda$ has the form $\phi\mu$ for some irreducible character $\mu\in\Irr(\overline{G})$. Since $\phi$ is linear, we have $|\phi(x)|=1$.  We also have $|\mu(x)|=|\mu(\overline{x})|=1$ as $\mu\in\Irr(\overline{G})$ and $\overline{x}$ is a root of unity in $\overline{G}$. Therefore \[|\chi(x)|=|\phi(x)\mu(x)|=|\phi(x)|\cdot|\mu(x)|=1,\] hence $x$ is a root of unity element of $G$.
\end{proof}

\section*{Acknowledgment} Part of this work was done while the third author was visiting the Vietnam Institute for Advanced Study in Mathematics (VIASM), and he thanks the VIASM for financial
support and hospitality.

\end{document}